\documentclass[a4paper,11pt]{amsart}

\usepackage{latexsym}
\usepackage{amssymb,amsmath}
\usepackage{graphics}
\usepackage{url}
\usepackage[vcentermath]{youngtab}
  
\usepackage{booktabs}  
\usepackage{caption}
\usepackage{tikz}
\usetikzlibrary{snakes}
\usetikzlibrary{arrows}
\tikzstyle{empty}=[circle,draw=black!80,thick]
\tikzstyle{emptyn}=[circle,draw=black!80,fill=white,scale=0.5] 
\tikzstyle{nero}=[circle,draw=black!80,fill=black!80,thick]

\newtheorem{teo}{Theorem}[section]
\newtheorem{corollary}[teo]{Corollary}

\newtheorem{proposition}[teo]{Proposition}

\newtheorem{lemma}[teo]{Lemma}

\theoremstyle{definition}
\newtheorem{definition}[teo]{Definition}
\newtheorem{remark}[teo]{Remark}
\newtheorem{example}[teo]{Example}

\newcommand{\N}{\mathbf{N}}
\newcommand{\Z}{\mathbf{Z}}

\renewcommand{\epsilon}{\varepsilon}

\DeclareMathOperator{\Ext}{Ext}

\DeclareMathOperator{\soc}{soc}
\renewcommand{\top}{\mathrm{top}\ }

\newcounter{thmlistcnt}
\newenvironment{thmlist}%
	{\setcounter{thmlistcnt}{0}%
	\begin{list}{\emph{(\roman{thmlistcnt})}}{%
		\usecounter{thmlistcnt}%
		\setlength{\topsep}{0pt}%
		\setlength{\leftmargin}{27pt}%
		\setlength{\itemsep}{0pt}%
		\setlength{\labelwidth}{20pt}
		\setlength{\itemindent}{0pt}}%
	}%
	{\end{list}}%

\newcommand{\ob}[1]{\langle#1\rangle}
\newcommand{\obb}[1]{\mathbf{\ob{#1}}}
\newcommand{\tb}[2]{\langle#1,#2\rangle}
\newcommand{\tbb}[2]{\mathbf{\tb{#1}{#2}}}

\newcommand{\m}{\!-\!}
\newcommand{\p}{\!+\!}

\title[Indecomposable summands of Foulkes modules]{Indecomposable summands of Foulkes modules}
\author{Eugenio Giannelli}
\address[E.~Giannelli]{Department of Mathematics, University of Kaiserslautern,
P.O. Box 3049, 67655 Kaiserslautern, Germany}
\email{gianelli@mathematik.uni-kl.de}

\author{Mark Wildon}
\address[M.~Wildon]{Department of Mathematics, Royal Holloway, University of London, United Kingdom}
\email{mark.wildon@rhul.ac.uk}

\thanks{The first author is supported by the London Mathematical Society Postdoctoral Mobility Grant PMG14-15 02.}


\begin{document}

\begin{abstract}
In this paper we study the modular structure of the permutation module $H^{(2^n)}$ of the symmetric
group $S_{2n}$ acting on set partitions of a set of size $2n$ into $n$ sets each of size $2$,
defined over a field of odd characteristic $p$. In particular we  
characterize the vertices of the indecomposable summands of $H^{(2^n)}$ 
and fully describe all of
its indecomposable summands that lie in blocks of $p$-weight at most two.
When $2n < 3p$ we show that there is a unique summand of $H^{(2^n)}$ in the principal
block of $S_{2n}$ and that this summand
exhibits many of the extensions between simple modules in its block.
\end{abstract}

\maketitle
\thispagestyle{empty}

\section{Introduction}

The symmetric group $S_{2n}$ acts on the 
collection of all set partitions of $\{1,\ldots, 2n\}$ 
into~$n$ sets each of size two. Let $H^{(2^n)}$ denote the corresponding permutation
representation of $S_{2n}$, defined over a field $F$ of odd prime characteristic $p$.
(Equivalently, $H^{(2^n)}$ is the $FS_{2n}$-module induced from the trivial representation
of the imprimitive wreath product $S_2 \wr S_n \le S_{2n}$.)
We call $H^{(2^n)}$ a \emph{Foulkes module}.  
In the main theorem of \cite{GiannelliWildonDec} the authors used results on
the indecomposable summands of Foulkes modules to determine certain
decomposition numbers of the symmetric group. In this paper we study the
structure of Foulkes modules more closely. In particular, we characterize the
vertices of the indecomposable summands of each $H^{(2^n)}$ and 
give a precise description of all summands in blocks of $p$-weight at most two.

Let $Q_t$ be a Sylow $p$-subgroup of $S_2 \wr S_{tp}$ and let $Q_0$ be the trivial group.
By \cite[Theorem 1.2]{GiannelliWildonDec}, if $U$ is an indecomposable summand of $H^{(2^n)}$
then $U$ has vertex $Q_t$ for some $t \in \N_0$. Our first main theorem gives the converse.

\begin{teo}\label{teo:vertices}
Let $n \in \N$. 
For all $t \in \N_0$ such that $t \le n/p$ there is an indecomposable summand of
$H^{(2^n)}$ with vertex $Q_t$.
\end{teo}


%
%

To state the second main theorem we need some more definitions from~\cite{GiannelliWildonDec}.
When
defined over a field
of characteristic zero, $H^{(2^n)}$ has 
ordinary character $\sum \chi^{2\lambda}$
where the sum is over all partitions $\lambda$ of $n$, and $2\lambda$ is the 
partition obtained from~$\lambda$ by doubling each part. (For an elegant proof of 
this fact with minimal prerequisites, see \cite{IRS}.) 
We say that such partitions are \emph{even}.
Given a $p$-core $\gamma$, let $w(\gamma)$ be the minimum number of $p$-hooks
that, when added to~$\gamma$, give an even partition.
Let $\mathcal{E}(\gamma)$
be the set of even partitions that can be obtained by adding $w(\gamma)$ $p$-hooks
to $\gamma$. For example, if $p=3$, then $w\bigl( (3,1) \bigr) = 2$ and $\mathcal{E}\bigl( (3,1)
\bigr) = \bigl\{ (6,4), (6,2,2), (4,4,2) \bigr\}$.
Let $B(\gamma,w)$
denote the $p$-block of $S_n$ with $p$-core $\gamma$ and $p$-weight $w$, where $n = |\gamma| + wp$.
As a convenient shorthand, we write $\nu \in B(\gamma,w)$ to mean that the partition $\nu$ has $p$-core
$\gamma$ and $p$-weight $w$. Let $S^\mu$ denote the Specht module labelled by the partition $\mu$.
For $\mu$ a $p$-regular partition,
 let $P^\mu$ 
denote the projective cover of the simple $FS_n$-module $D^\mu$, defined in
\cite[Corollary 12.2]{James} as the unique top composition factor of $S^\mu$.
Finally let $d_{\mu\nu} = [S^\mu : D^\nu]$.

\begin{teo}\label{teo:lowweight}
Let $n \in \N$. Let $\gamma$ be a $p$-core.
\begin{thmlist}
\item There is a summand of $H^{(2^n)}$ in
$B(\gamma,0)$ if and only if $\gamma$ is even. In this case the unique summand
is the simple projective Specht module~$S^\gamma$.
\item There is a summand of $H^{(2^n)}$ in $B(\gamma,1)$ if and only if 
$w(\gamma) = 1$. In this case $\mathcal{E}(\gamma) = \{2\lambda, 2\lambda'\}$
for partitions $\lambda$, $\lambda'$ with $\lambda \lhd \lambda'$,
and the unique summand is $P^{2\lambda'}$.
\item There is a summand of $H^{(2^n)}$ in $B(\gamma,2)$ if and only if
$w(\gamma) = 0$ or $w(\gamma) = 2$. If $w(\gamma) = 2$ the unique summand is 
$P^{2\lambda}$, where $2\lambda$ is the unique
maximal element of $\mathcal{E}(\gamma)$. If $w(\gamma) = 0$ the unique summand has
vertex $Q_1$ and its Green correspondent is $P \otimes S^\gamma$ as a representation of
$N_{S_{2n}}(Q_1)/Q_1 \cong N_{S_{2p}}(Q_1)/Q_1 \times S_{n-2p}$, where $P$ is
the projective cover of the trivial $FN_{S_{2p}}(Q_1)/Q_1$-module.
\end{thmlist}
\end{teo}

If $t = \lfloor n/p\rfloor$ then $Q_t$ is a Sylow $p$-subgroup of $S_2 \wr S_n$.
The projective cover of the trivial representation of $N_{S_{2n}}(Q_t)/Q_t$ has
vertex $Q_t$ as a representation of $N_{S_{2n}}(Q_t)$; its Green correspondent
is a summand of $H^{(2^n)}$ with vertex $Q_t$ lying in the principal block of $S_{2n}$. 
This summand is an example of a Scott module: see \cite{GreenScott} or \cite{BroueBrauer} for
their definition and basic properties. Our third main theorem describes these summands when $2n < 3p$.

\begin{teo}\label{teo:Scott}
Let $2k < p$ and let $2n = 2p + 2k$.
There is a unique summand~$U$ of $H^{(2^n)}$ 
in the principal block of $S_{2n}$. This summand is the Scott module with vertex~$Q_1$.
The module $U$ has three Loewy layers and
\[ \soc U \cong \top U \cong \bigoplus_{2 \nu \in B( (2k), 2) \atop 2\nu\not= (2k,2^p)} D^{2\nu}. \]
The Loewy layers of $U$ are shown in Figure~1 overleaf, where each edge shows
an element of $\Ext^1$ exhibited by this module.
\end{teo}

\begin{figure}[h]
\begin{center}
\makebox[\textwidth]{
\begin{tikzpicture}[scale=1, x=1cm, y=1cm]
\tikzstyle{nero} = [draw, circle, minimum size=4pt, inner sep=0mm, fill]

\renewcommand{\scriptstyle}[1]{\scalebox{0.9}{$#1$}}
\def\x{6.5}\def\y{13.0}\def\v{0}
\newcommand{\s}[1]{\scalebox{0.9}{#1}}

\foreach \w in {0,\x,\y} {
\draw (0+\w,0)--(2+\w,-2);
\draw (0+\w,-2)--(2+\w,0);
}

\draw (2,0)--(3.5,-1.5);
\draw (2,-2)--(3.5,-0.5);

\draw (\x,0)--(\x-1.5,-1.5);
\draw (\x,-2)--(\x-1.5,-0.5);
\draw (\x+2,0)--(\x+3.5,-1.5);
\draw (\x+2,-2)--(\x+3.5,-0.5);

\draw (\y+2,0)--(\y+3,-1);
\draw (\y+2,-2)--(\y+3,-1);
\draw (\y,0)--(\y-1.5,-1.5);
\draw (\y,-2)--(\y-1.5,-0.5);

\foreach \z in {0,-2} {
\node at (0,\z) [fill=white] {\s{$D^{\ob{2k}}$}};
\node at (2,\z) [fill=white] {\s{$D^{\tb{p-2}{p-1}}$}};
\node at (\x,\z) [fill=white] {\s{$D^{\tb{2k+1}{2k+2}}$}};
\node at (\x+2.5,\z) [fill=white] {\s{$D^{\tb{2k-2}{2k-1}}$}};
\node at (\y,\z) [fill=white] {\s{$D^{\tb{4}{5}}$}};
\node at (\y+2,\z) [fill=white] {\s{$D^{\tb{2}{3}}$}};
\node at (4.25,\z) {$\cdots$};
\node at (10.875,\z) {$\cdots$};
}

\node at (1,-1) [fill=white] {\s{$D^{\tb{2k}{p-2}}$}};
\node at (3,-1) [fill=white] {\s{$D^{\tb{p-4}{p-2}}$}};
\node at (\x+1,-1) [fill=white] {\s{$D^{\tb{2k-2}{2k+1}}$}};
\node at (\x+3,-1) [fill=white] {\s{$D^{\tb{2k-4}{2k-2}}$}};
\node at (\x-1,-1) [fill=white] {\s{$D^{\tb{2k+1}{2k+3}}$}};

\node at (\y-1,-1) [fill=white] {\s{$D^{\tb{4}{6}}$}};
\node at (\y+1,-1) [fill=white] {\s{$D^{\tb{2}{4}}$}};
\node at (\y+3,-1) [fill=white] {\s{$D^{\tb{0}{2}}$}};

\end{tikzpicture}}
\end{center}

\medskip
\caption{\small The three Loewy layers of the unique principal block summand of $H^{(2^n)}$
when $n = 2(p+k)$ and $2 < 2k < p-1$. 
The labels of
simple modules are defined in Section 7. If $2k \in \{0,2,p-1\}$ then the structure
is the same but minor changes must be made to the labels: see Figure 5 in Section 7.}

\end{figure}
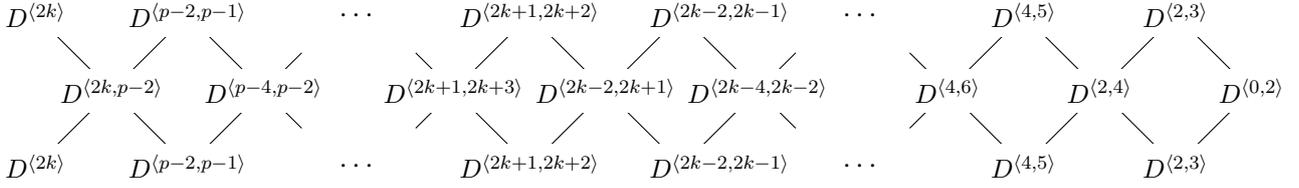

Our final theorem counts the summands of $H^{(2^n)}$ which, like the  Scott module summand,
have the largest possible vertex.

\begin{teo}\label{teo:maximalvertex}
Let $t = \lfloor n/p \rfloor$ and let $r = n-tp$. The number of indecomposable
summands of $H^{(2^n)}$ with vertex $Q_t$ is equal to the number of 
$p$-core partitions that can be obtained by removing $p$-hooks from even partitions of~$2r$.
Each such summand lies in a different block of $S_{2n}$.
\end{teo}

In particular we note that if $0 \le 2r < p$ then the number of indecomposable
summands of $H^{(2^n)}$ with vertex $Q_t$ is simply the number of even partitions of $2r$.
Another easy corollary of our theorems is the following.

\begin{corollary}\label{cor:indec}
The unique non-trivial indecomposable Foulkes module is $H^{(2^2)}$ when $p=3$.
\end{corollary}

\begin{proof}
By Theorem~\ref{teo:vertices} there are summands of $H^{(2^n)}$ with vertices $Q_0$ and $Q_1$
whenever $n \ge p$. When $4 \le n < p$  the partitions $(2n)$, $(2n-2,2)$
and $(2n-4,4)$ are not all in the same block, so $H^{(2^n)}$ has summands in two
different blocks.
By Theorem~\ref{teo:lowweight},
if $p=3$ then $H^{(2^2)}$ is indecomposable; if $p > 3$ then
 $H^{(2^2)} = S^{(4)} \oplus S^{(2,2)}$. If $p=3$ 
then $H^{(2^3)}$ has $S^{(4,2)}$ as the unique summand in its block;
if $p=5$ then
$H^{(2^3)}$ has $S^{(2,2,2)}$ as the unique summand in its block;
if $p > 5$ then $H^{(2^3)} = S^{(6)} \oplus S^{(4,2)} \oplus S^{(2,2,2)}$.
\end{proof}


The
results in this paper 
show that the behaviour of Foulkes modules 
is very different from the  better
studied Young permutation modules  (see, for instance, 
\cite{ErdmannYoung}, \cite{CGill}, \cite{Henke} and \cite{JamesYoung}). In particular,
by Theorem~\ref{teo:lowweight}, 
each Foulkes module has at most one summand in each block
of weight at most two. The Young permutation modules $M^{(r-s,s)}$
also have multiplicity-free ordinary characters,
but it follows easily from \cite[Theorem~3.3]{Henke} that
if $p$ is odd, $0 \le c < p$ and $p \le j \le p + c/2$ then
 $M^{(2p+c-j,j)}$ always has two summands, namely the Young modules
$Y^{(2p+c)}$ and $Y^{(p+c,p)}$, in the principal block of $S_{2p+c}$.
 By Corollary~\ref{cor:indec} there is only one non-trivial indecomposable Foulkes module,
  whereas $M^{(p-1,1)}$ is always
indecomposable. (The indecomposable Young permutation modules
 are classified in \cite[Theorem~2]{CGill}.)
It would be interesting to know the decomposition
of further permutation modules for symmetric groups in prime characteristic.

\subsection*{Outline}
In \S 2 we give the prerequisite material from \cite{GiannelliWildonDec}.
In \S 3 we find 
$w(\gamma)$ for a special class of $p$-cores $\gamma$. This result is used in the proof
of Theorem~\ref{teo:vertices}.
In \S 4 we recall Richards' results \cite{Richards}
on decomposition numbers in blocks of $p$-weight~$2$, 
and show that even partitions of $p$-weight~$2$ have
a surprisingly simple characterization using his definitions. These results
are used in the proof of Theorem~\ref{teo:lowweight}(iii) and Theorem~\ref{teo:Scott}.
In \S 5, \S 6 and \S 7
we prove Theorem~\ref{teo:vertices},  Theorem~\ref{teo:lowweight} and
Theorem~\ref{teo:Scott} respectively. Theorem~\ref{teo:maximalvertex} is obtained as a corollary of Theorem ~\ref{teo:lowweight} in \S 6.
We end in \S 8 with an example to illustrate these theorems.
The proofs in \S 3 and \S 4 are the most 
technical in this paper; we suggest the reader returns to the proofs after seeing the applications.

\section{Prerequisites}
The following results have short proofs using \cite{GiannelliWildonDec}. 
For background on local representation theory we refer the reader to \cite{Alperin}.

\begin{proposition}\label{prop:Green}
Let $\gamma$ be a $p$-core partition and let $2n = |\gamma| + pw$.
The Green correspondence induces a bijection between the indecomposable summands of $H^{(2^n)}$ in
$B(\gamma, w)$ with vertex $Q_t$ and the indecomposable projective summands of $H^{(2^{n-tp})}$
in $B(\gamma, w-2t)$.
\end{proposition}

\begin{proof}
Let $U$ be an indecomposable non-projective summand of $H^{(2^n)}$ in $B(\gamma, w)$. 
By \cite[Theorem 1.2]{GiannelliWildonDec}, $U$ has $Q_t$ as a vertex for some $t \in \N$ and
its Green correspondent has a tensor
factorization $V \boxtimes W$ as a representation of $(N_{S_{2tp}}(Q_t)/Q_t) \times S_{2(n-tp)}$.
Here $V$ and $W$ are projective and $W$ is an indecomposable summand of $H^{(2^{n-tp})}$.
By \cite[Theorem 2.7]{GiannelliWildonDec}, which is proved using the results in \cite{BroueBrauer},
$W$ lies in the block $B(\gamma, w-2t)$ of $S_{2(n-tp)}$. The map sending $U$ to $W$ therefore has
the required properties.
\end{proof}

\begin{proposition}\label{prop:proj}
Let $\gamma$ be a $p$-core partition and let $2n = |\gamma| + pw$.
\begin{thmlist}
\item All the summands of $H^{(2^n)}$ in $B\bigl( \gamma,w(\gamma) \bigr)$
are projective. 
\item If $2\lambda$ is a maximal element of $\mathcal{E}(\gamma)$
then $\lambda$ is $p$-regular and $P^{2\lambda}$ is a summand of $H^{(2^n)}$ lying in $B\bigl(
\gamma, w(\gamma) \bigr)$.
\item If $P$ is an indecomposable projective summand of $H^{(2^n)}$ in $B(\gamma,w(\gamma))$ 
then $P = P^{2\nu}$ for some $2\nu \in \mathcal{E}(\gamma)$ and
the ordinary character of $P$ (defined using Brauer reciprocity)
is $\sum_{2\lambda \in E} \chi^{2\lambda}$
for some $E \subseteq \mathcal{E}(\gamma)$.
\end{thmlist}

\end{proposition}

\begin{proof}
Part (i) is Proposition 5.1 of \cite{GiannelliWildonDec}. Since this is a key result in this paper,
we briefly sketch the proof: if $U$ is a non-projective summand in $B\bigl( \gamma,w(\gamma) \bigr)$
then, by Proposition~\ref{prop:Green}, there is a projective summand of~$H^{(2^{n-tp})}$
in $B\bigl( \gamma,w(\gamma)-2t \bigr)$ for some $t \in \N$. But this implies that
there is an even partition
with $p$-core $\gamma$ and $p$-weight $w(\gamma)-2t$, contradicting the definition of $w(\gamma)$.
Part (ii) is Proposition~1.3 of \cite{GiannelliWildonDec}. 
Part (iii) follows immediately from ($\dagger$) in the proof of this proposition.
\end{proof}

We use the following proposition in the proof of Theorem~\ref{teo:lowweight} to
show that, if $w \le 2$, then 
the summand identified in Proposition~\ref{prop:proj}(ii) is the unique
summand of $H^{(2^n)}$ lying in its block. 

\begin{proposition}\label{prop:Cartan}
Let $\gamma$ be a $p$-core partition of $n$. If $|\mathcal{E}(\gamma)| < 2w(\gamma)+1$ then
there is a unique summand of $H^{(2^n)}$ in $B(\gamma,w(\gamma))$.
\end{proposition}

\begin{proof}
It was proved independently in 
\cite[Theorem~2.8]{Richards} and \cite[Proposition~4.6(i)]{BessenrodtUno}
that if $\nu$ is a $p$-regular partition of $n$ lying in a block
of $p$-weight~$w$ then $d_{\mu \nu} > 0$ for at least $w+1$ partitions $\mu$.
By Proposition~\ref{prop:proj}(i), all the summands of $H^{(2^n)}$ in $B\bigl(\gamma,w(\gamma)\bigr)$
are projective. By Brauer reciprocity, the ordinary character of each
such summand contains at least $w+1$ distinct
irreducible characters. The result now follows from Proposition~\ref{prop:proj}(iii).
\end{proof}

\section{Even partitions and the abacus}

We make extensive use of James' abacus notation, as defined in \cite[page 78]{JK}.
We number the abacus runners from $0$ to $p-1$. 
By convention, our abaci have beads in all strictly negative positions.
We say that beads before the first space are \emph{initial}.
For $r \in \N_0$,
we define \emph{row} $r$ to consist of positions
$pr,\ldots, pr+p-1$.
 By a \emph{single-step move} 
we mean a move
of a bead into a  space immediately below it.

It is simple to recognise even partitions on the abacus.

\begin{definition}
Let $A$ be an abacus. We say that beads in positions
$\beta$ and $\beta'$ of $A$ where $\beta < \beta'$ are \emph{adjacent beads} and form a \emph{gap} if
there are spaces in positions $\beta+1, \ldots, \beta'-1$. This gap
is \emph{odd} if  $\beta'-\beta$ is even and \emph{even} if $\beta'-\beta$ is odd.
If $\beta'-\beta=1$, then we say that the beads are in \emph{adjacent positions}.
\end{definition}

\begin{lemma}
Let $A$ be an abacus representing a partition $\nu$. Then $\nu$ is even if and only if
$A$ has no odd gaps.\hfill$\qed$
\end{lemma}

The following  proposition is helpful in determining $w(\gamma)$. Informally, it states
that certain beads on an abacus representing $\gamma$ need not be moved to find
a partition in $\mathcal{E}(\gamma)$. 

\begin{proposition}\label{prop:beadsfromabove}
Let $p$ be odd and let $\gamma$ be a $p$-core represented by an abacus~$A$
with first space in position $0$. Let $B$ be an abacus representing a partition $\lambda \in \mathcal{E}(\gamma)$
obtained by a sequence of single-step moves on~$A$. 
If $x \le -2$ then $B$ has a bead in position $x$.
\end{proposition}

\begin{proof}
Suppose, for a contradiction, that there exists $x < -1$ such that $B$ has a space in position $x$.
Choose $x$ minimal with this property. Thus if $y < x$ then the bead in position $y$ of $A$ remains fixed,
and the bead in position $x$ is moved. To avoid an odd gap, the bead in position $x+1$ of $A$ is also moved.
By increasing the number of initial beads in $A$ we may assume that $x = pr$ for some $r \in \Z$.
Therefore positions $pr$ and $pr+1$ in $B$ have spaces.

Let $b$ be the highest bead on runner~$0$ of $B$ below position~$pr$,
and let $c$ be the highest bead on runner~$1$ of~$B$ below position~\hbox{$pr+1$}.
If $b$ and $c$ are on the same row then 
we may move $b$
up to position~$pr$ and $c$ up to position $pr+1$ without
changing the parity of any gap in~$B$. (This holds because there is a bead in position $pr-1$,
and the gap between this bead and the adjacent bead in $B$ is even.)
We therefore obtain  an even partition with $p$-core
$\gamma$ and strictly smaller $p$-weight, a contradiction.

Assume that $c$ is on a strictly lower row than $b$; the other case
is dealt with symmetrically.
Let $b'$ be the first bead in a greater numbered position than $b$,
and let $c'$ be the first bead in a greater numbered position than $c$.
Suppose that $b'$ is on runner $u$ 
and~$c'$ is on runner $v$.
There are at most $2(p-1)$ spaces 
between $b$ and $b'$ since otherwise  $b'$ can be moved two rows up without changing the
parity of any gaps. 
Hence either $b'$ is on the same row as $b$,
in which case~$u$ is odd, or $b'$ is on the row directly below,
in which case $u$ is even. A similar remark
applies to $c$ and $c'$, swapping odd and even, with one extra case
when $v = 0$, in which case $c'$ is two rows below~$c$.

We claim that there is a
sequence of bead moves using the beads $b$ and~$b'$ and beads on runners $u$ and $v$ that
gives a new abacus $B^\star$, representing a partition of strictly smaller $p$-weight than~$\lambda$,
such that $B^\star$ 
has no odd gaps. 
This can be shown by considering the four cases
for the parity of $u$ and~$v$, splitting the different parity
cases into the subcases $u<v$ and $u>v$, as shown in Figure~2,
and the equal parity cases into subcases for $u=0$ and $u\not= 0$,
as shown in Figure~3. The indicated bead moves deal correctly
with the exceptional cases when $u=0$, $v=0$ or $v=1$ and two runners coincide. (Note that $u=1$ 
is impossible because then $b'=c$ and $b$ and $c$ are in the same row.)
If $v=0$ and $u$ is odd then $c'$ is two rows below $c$, 
the bead moves in the top right diagram in Figure~2 apply and a bead is moved into
the position vacated by~$b$.

We give full details for the case 
where $B$ is the top-left abacus
in Figure~2. Thus
$c$ is on a strictly lower row than $b$, and $b'$ and $c'$ are in the same rows as $b$ and $c$, respectively. 
Let $s$ be the row of bead $b$, $t$ be the row of bead $c$ and let $w$
be the $p$-weight of the partition $\lambda$. 
Let $B'$ be the abacus obtained from $B$ by moving beads~$b$ and~$c$ up to positions $pr$ and $pr+1$ respectively by a sequence of single-step bead moves.
 Notice that $B'$ has no odd gaps in rows $1,\ldots,s-1$, 
 and that $B'$ represents a partition of $p$-weight $w-(s-r)-(t-r)$.
Let $B''$ be the abacus obtained from $B'$ by making the unique series of single-step moves of beads on runner $u$ that has the final effect of swapping bead $b'$ with the space
in position $pt+u$, in row $t$ on runner $u$. (The bead that
ends in position $pt+u$ in $B''$ is therefore the lowest bead on runner~$u$ of $B'$
above this position.) It is easily seen that
$B''$ has no odd gaps and represents a partition of $p$-weight 
\[ w-(s-r)-(t-r)+(t-s)=w-2(s-r). \]
This is clearly strictly less than $w$, the $p$-weight of $\lambda$, as required.
\end{proof}

\begin{figure}[t]

\includegraphics{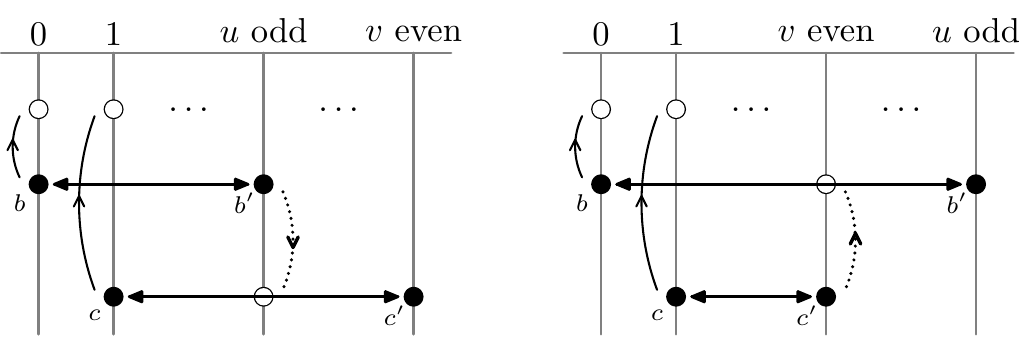}

\medskip
\includegraphics{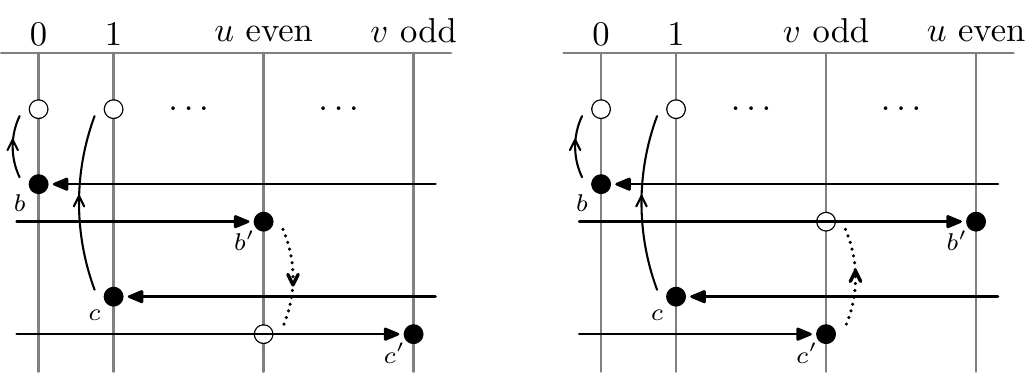}
\caption{\small Bead moves on the abacus $B$ used in
the proof of Proposition~\ref{prop:beadsfromabove} when~$u$ and $v$ have different parity. 
White circles show
spaces in $B$ that are present by hypothesis and are
the target of bead moves; the two adjacent spaces are positions~$pr$
and $pr+1$. Thick arrows show even gaps in~$B$. 
Solid curved arrows show moves of the beads $b$ and $c$ on runners $0$ and $1$, respectively.
On runners $u$ and~$v$ dotted curved arrows show the overall effect
of a sequence of single-step upward or downward moves using the marked bead and the
 beads between the marked bead and the target space.}
\end{figure}

\begin{figure}[h]

\includegraphics{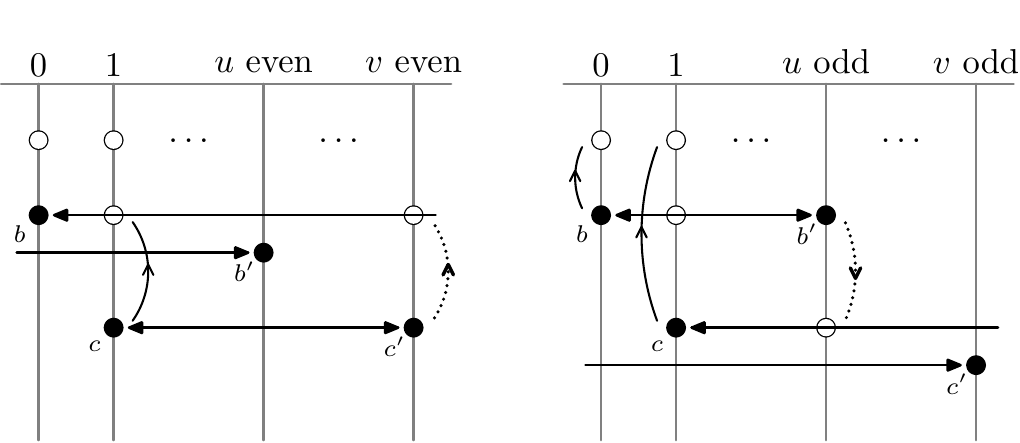}

\medskip
\includegraphics{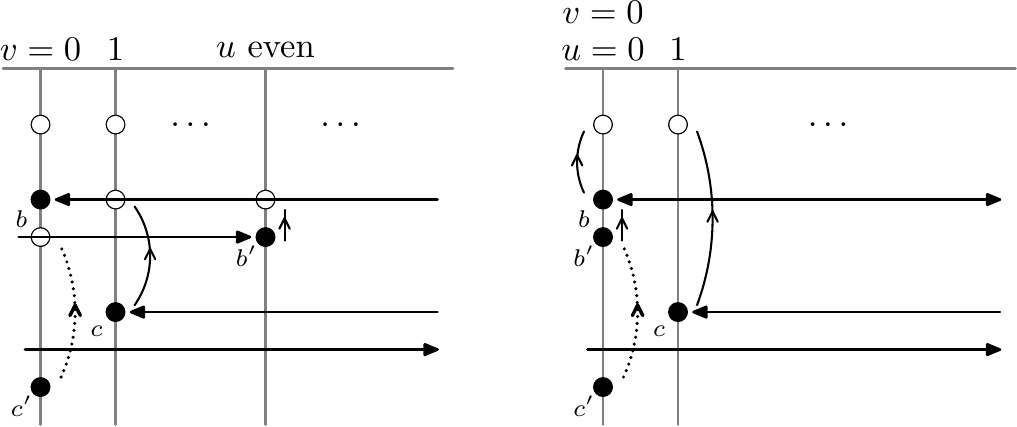}

\caption{\small Bead moves on the abacus $B$ used in
the proof of Proposition~\ref{prop:beadsfromabove} when $u$ and $v$ have equal parity.
The conventions for arrows are as described in the caption of Figure~2.
The cases $u=0$, $v\not=0$ and $u=v=0$ are dealt with separately: in the latter case 
at least four
beads move, bead~$b'$ is moved
to the position vacated by bead~$b$ and a bead is moved into
the position vacated by bead~$b'$.
}
\end{figure}

%

Despite its special nature, the following result appears to require the full power 
of Proposition~\ref{prop:beadsfromabove}. 

\begin{proposition}\label{prop:core}
Let $\gamma$ be the $p$-core represented by an abacus having $t$ beads on runner $u$
and $t'$ beads on runner $u'$, where $u < u'$, and no beads on any other runner.
\begin{thmlist}
\item If $u$ and $u' - u$ are both even then $w(\gamma) = tt'$ and a partition
in $\mathcal{E}(\gamma)$ can be obtained by moving each of the $t'$ beads on runner $u'$
down $t$ steps.
\item If $u$ is odd and $u' = p-1$ then $w(\gamma) = t(t'+1)$ and a partition
in $\mathcal{E}(\gamma)$ 
can be obtained by moving each of the $t$ beads on runner $u$ down $t'+1$ steps.
\end{thmlist}
\end{proposition}

\begin{proof}
We prove case (i) and then indicate the changes needed for (ii).
The $p$-core $\gamma$ is represented by the abacus $A$ shown below
in which the first space is in position $0$ and $p$ initial beads occupy
positions $-p, \ldots, -1$.
(If $t > t'$ then  runners $u$ and $u'$ should be swapped.)
Let $b$ and $b'$ be the beads in positions $u$ and $u'$ respectively and
let $c$ and $c'$ be the lowest beads on runners $u$ and $u'$ respectively.

\begin{center}
\begin{tikzpicture}[scale=1, x=0.5cm, y=0.5cm, every node/.style={transform shape}]
\tikzstyle{empty}=[draw, circle, minimum size=5pt, inner sep=0mm]
\tikzstyle{nero} =[draw, circle, minimum size=4pt, inner sep=0mm, fill]

\draw (-0.5,0.5)--(14.5,0.5);
\node at (2,-3.25) {$\cdots$}; 
\node at (7,-3.25) {$\cdots$}; 
\node at (12,-3.25) {$\cdots$}; 

\node at (2,1) {$\cdots$}; 
\node at (7,1) {$\cdots$}; 
\node at (12,1) {$\cdots$}; 

\node at (2,-5.5) {$\cdots$}; 
\node at (7,-5.5) {$\cdots$}; 
\node at (12,-5.5) {$\cdots$}; 

\node at (0,0) [nero] {}; \node[yshift=15pt] at (0,0) {$\scriptstyle 0$};
\node at (2,0) {$\cdots$}; 
\node at (4,0) [nero] {};  \node[yshift=15pt] at (4,0) {$\scriptstyle u$};
\node at (5,0) [nero] {}; \node[xshift=3pt,yshift=15pt] at (5,0) {$\scriptstyle u+1$};
\node at (7,0) {$\cdots$}; 
\node at (9,0) [nero] {}; \node[xshift=-3pt,yshift=15pt] at (9,0) {$\scriptstyle u'-1$};
\node at (10,0) [nero] {}; \node[yshift=15pt] at (10,0) {$\scriptstyle u'$};
\node at (12,0) {$\cdots$};
\node at (14,0) [nero] {}; \node[yshift=15pt] at (14,0) {$\scriptstyle p-1$};

\node at (0,-1) [empty] {};
\node at (2,-1) {$\cdots$};
\node at (4,-1) [nero] {};
\node at (5,-1) [empty] {};
\node at (7,-1) {$\cdots$};
\node at (9,-1) [empty] {};
\node at (10,-1) [nero] {};
\node at (12,-1) {$\cdots$};
\node at (14,-1) [empty] {};

\node at (0,-2) {$\vdots$};
\node at (4,-2) {$\vdots$};
\node at (5,-2) {$\vdots$};
\node at (9,-2) {$\vdots$};
\node at (10,-2) {$\vdots$};
\node at (14,-2) {$\vdots$};

\node at (0,-3.25) [empty] {};
\node at (4,-3.25) [nero] {};
\node at (5,-3.25) [empty] {};
\node at (9,-3.25) [empty] {};
\node at (10,-3.25) [nero] {};
\node at (14,-3.25) [empty] {};

\node at (0,-4.25) {$\vdots$};
\node at (4,-4.25) {$\vdots$};
\node at (5,-4.25) {$\vdots$};
\node at (9,-4.25) {$\vdots$};
\node at (10,-4.25) {$\vdots$};
\node at (14,-4.25) {$\vdots$};

\node at (0,-5.5) [empty] {};
\node at (4,-5.5) [empty] {};
\node at (5,-5.5) [empty] {};
\node at (9,-5.5) [empty] {};
\node at (10,-5.5) [nero] {};
\node at (14,-5.5) [empty] {};

\node[right] at (14,0) {$\scriptstyle d$};
\node[left] at (4,-1) {$\scriptstyle  b$};
\node[left] at (4,-3.25) {$\scriptstyle  c$};
\node[right] at (10,-5.5) {\raisebox{3.5pt}{$\scriptstyle  c'$}};
\node[right] at (10,-1) {\raisebox{1pt}{$\scriptstyle  b'$}};
\end{tikzpicture}
\end{center}

It is clear that $t$ single-step moves of each of the $t'$ beads between $b'$ and~$c'$ (inclusive,
starting by moving bead $c'$)
give an abacus representing an even partition. 
It remains
to show that $tt'$ moves are necessary.

Let $B$ be an abacus representing a partition in $\mathcal{E}(\gamma)$
obtained from $A$ by a fixed sequence of single-step moves.
No bead in a negative position other than $-1$ is moved, by Proposition~\ref{prop:beadsfromabove}.
If the bead labelled $d$ in position $-1$ is moved, then to avoid an odd gap,
the beads $b$ and $b'$ must also be moved. Raising $d$ back to position $-1$
and raising all the beads in positive positions in $B$ up one row now gives an abacus
representing an even partition of smaller $p$-weight.
Hence the odd gaps between the beads on runners $u$ and $u'$ are removed by moving beads
only on these runners. 

Either bead $c$ or bead $c'$ is the lowest bead in $B$.
If $c$ is the lowest then~$c$ 
must have been moved down at least $t'$ steps from its position in $A$, for otherwise
there are $t+t'$ beads in at most $t+t'-1$ rows of $B$, and so some
row is of the form 
\begin{center}
\begin{tikzpicture}[scale=1, x=0.5cm, y=0.5cm, every node/.style={transform shape}]
\tikzstyle{empty}=[draw, circle, minimum size=5pt, inner sep=0mm]
\tikzstyle{nero} =[draw, circle, minimum size=4pt, inner sep=0mm, fill]
\draw (-0.5,0.5)--(14.5,0.5);
 \node[yshift=15pt] at (0,0) {$\scriptstyle 0$};
\node at (2,1) {$\cdots$}; 
 \node[yshift=15pt] at (4,0) {$\scriptstyle u$};
\node[xshift=3pt,yshift=15pt] at (5,0) {$\scriptstyle u+1$};
\node at (7,1) {$\cdots$}; 
 \node[xshift=-3pt,yshift=15pt] at (9,0) {$\scriptstyle u'-1$};
\node[yshift=15pt] at (10,0) {$\scriptstyle u'$};
\node at (12,1) {$\cdots$};
\node[yshift=15pt] at (14,0) {$\scriptstyle p-1$};
\node at (0,0) [empty] {};
\node at (2,0) {$\cdots$};
\node at (4,0) [nero] {};
\node at (5,0) [empty] {};
\node at (7,0) {$\cdots$};
\node at (9,0) [empty] {};
\node at (10,0) [nero] {};
\node at (12,0) {$\cdots$};
\node at (14,0) [empty] {};
\end{tikzpicture}
\end{center}
in which the beads on runners $u$ and $u'$ form an odd gap, a contradiction.
By induction $(t-1)t'$ moves are necessary to remove all odd gaps from
the abacus obtained from $A$ by deleting $c$. Hence at least $(t-1)t' + t' = tt'$
moves were made to reach $B$ from $A$. The proof is similar if $c'$ is the lowest bead in $B$.

In case (ii) a similar argument applies. Since the bead in position $-1$ is on runner $p-1$,
it follows immediately from Proposition~\ref{prop:beadsfromabove} that only beads on runner
$u$ and runner $p-1$ are moved.
\end{proof}

\section{Even partitions and blocks of $p$-weight $2$}

Let $\gamma$ be a $p$-core and 
let $\nu \in B(\gamma, 2)$. Thus $\gamma$ is obtained from $\nu$
by removing two rim $p$-hooks. In \cite[page 397]{Richards}, 
Richards defines $\delta(\nu)$ to be the absolute
value of the difference of the leg-lengths of these rim $p$-hooks. (This difference is independent
of the choice of rim $p$-hooks.) A partition~$\nu$ such that $\delta(\nu) = 0$ either
has two rim $p$-hooks, or a unique rim $p$-hook and a rim $2p$-hook. In the former
case $\nu$ is \emph{black} if the larger leg-length is even, 
and \emph{white} if it is odd. In the latter case
$\nu$ is \emph{black} if the leg-length of the rim $2p$-hook is congruent to $0$ or $3$ modulo $4$,
and otherwise \emph{white}. 
If $\nu$ is $p$-regular,
let $\nu^{\circ}$ denote the largest partition such that $\nu^\circ \unlhd\hskip1pt \nu$
and $\delta(\nu^\circ) = \delta(\nu)$, and if $\delta(\nu) = 0$,
such that $\nu$ and $\nu^\circ$ have the same colour. (Thus $\nu^\circ = \nu^{\diamond'}$ in
Richards' notation.)

The main result of \cite{Richards} can be stated as follows.

\begin{teo}[Theorem 4.4 in \protect{\cite{Richards}}]\label{teo:Richards}
Let $\gamma$ be a $p$-core and let $\nu \in B(\gamma,2)$ be $p$-regular.
Then $d_{\nu \nu} = d_{\nu^\circ \nu} = 1$. If $\mu \in B(\gamma,2)$
and $\mu\not\in \{\nu,\nu^\circ\}$
then $d_{\mu \nu} = 1$ if $\nu^\circ \lhd \mu \lhd \nu$ and
 $\delta(\nu) - \delta(\mu) \in \{ 1,-1\}$; otherwise $d_{\mu\nu} = 0$.
\end{teo}

In the proof of Theorem~\ref{teo:lowweight}(iii) we also need part (ii) of
the following lemma. In the proof we use that the leg-length
of a rim $p$-hook
corresponding to a single-step upward move of a bead $b$ 
in position $\beta$ on an abacus 
is the number of beads in positions $\beta-(p-1),\ldots, \beta-1$.

\begin{lemma}\label{lemma:delta0}
Let $\gamma$ be an even $p$-core represented by an abacus $A$ and let $\nu \in B(\gamma,2)$.
\begin{thmlist}
\item If $\nu$ is even then $\nu$ is obtained
either by single-step moves of two adjacent beads in $A$, or by two single-step moves
of the final bead in~$A$. 
\item If $\nu$ is even then $\delta(\nu) = 0$.
\end{thmlist}
\end{lemma}

\begin{proof}
Suppose bead $b$ on $A$ is one of the beads moved to obtain an abacus representing $\nu$. 
If there is a bead, say $c$, after $b$, then
a new odd gap is created between the bead before $b$ in the abacus and bead $c$.
Hence one of these beads must be moved. This proves~(i).
If the final bead is moved twice then the leg-length of the corresponding rim $2p$-hook
is $0$, so $\delta(\nu) = 0$. Otherwise  adjacent beads $b$ and $b'$
are moved, where $b$ and $b'$ are in positions $\beta$ and $\beta'$ with $\beta < \beta'$.
Let $\mu$ be obtained from $\nu$ 
by removing the rim $p$-hook corresponding to $b'$.
By
the observation before the lemma,
 the leg-lengths of the rim $p$-hooks in $\nu$ and $\mu$ 
corresponding to bead $b'$ and $b$, respectively, are
equal to the number of beads in $A$ in positions $\{\beta, \ldots, \beta+(p-1)\}$.
Hence $\delta(\nu) = 0$, as required for (ii).
\end{proof}

It is possible to sharpen Lemma~\ref{lemma:delta0}(ii) so that the converse also holds.
For this we need one more statistic. If $\nu \in B(\gamma,2)$ 
is obtained by single-step moves of 
beads $b$ and $b'$ on an abacus representing $\gamma$, let $\Delta(\nu)$ be
the number of beads strictly between $b$ and $b'$. (This is clearly independent of the choice
of abacus.) If $\nu$ has a rim $2p$-hook
then let $\Delta(\nu) = 0$. The proof of Lemma~\ref{lemma:delta0} shows that
if $\Delta(\nu) = 0$ then $\delta(\nu) = 0$.

\begin{proposition}\label{prop:Delta0}
Let $\gamma$ be an even $p$-core and let $\nu \in B(\gamma,2)$. Then $\nu$ is even
if and only if $\Delta(\nu) = 0$ and $\nu$ is black.
\end{proposition}

\begin{proof}
Let $A$ be an abacus representing $\gamma$. Let $\nu$ be a partition obtained from $A$
by two single-step moves, represented by the abacus $B$. We consider two cases.

\subsubsection*{Case 1: Distinct beads}
Suppose that distinct beads $b$ and $b'$ are moved to obtain~$B$.
Let $b$ and $b'$ be in positions $\beta$ and $\beta'$ of $A$, respectively, where $\beta < \beta'$.
If $\nu$ is even then, by Lemma~\ref{lemma:delta0}(i),  $b$ and $b'$ are adjacent and so $\Delta(\nu) = 0$.
Conversely, if $b$ and $b'$ are adjacent then $\Delta(\nu) = 0$. It remains to
show that, if $b$ and $b'$ are adjacent, then $\nu$ is even
if and only if $\nu$ is black.

The rim $p$-hook in $\nu$ with the longer leg-length 
corresponds to bead $b'$ in~$B$. Its leg-length, $\ell$ say, is 
the number of beads in positions
$\beta'+1,\ldots,\beta+p$ of~$B$. Let $a$ be the first bead before
position $\beta$ in $A$. 
Let $e$ be the first bead before position $\beta+p$ in $B$
and let $a'$ be the first bead after position $\beta'+p$ in $B$; if there
is no such bead let $a' = b'$.

If bead $e$ is bead $a$ then there are no beads except for $b$ and $b'$ between $a$
and $a'$ in $B$. In this case $\ell = 1$ and so $\nu$ is white, and, because
of the gap in~$B$ between beads $a$ and $b$, $\nu$ is not even.
In the remaining case bead~$e$
is in a position $\epsilon$ strictly between $\beta'$ and $\beta+p$. By
inserting initial beads if necessary we may assume that bead $e$ is on runner $0$,
and so the relevant rows of $A$ are as shown in the abacus below, where arrows indicate the 
gaps in $A$ and $B$ affected by the  moves of $b$ and $b'$.
\begin{center}
\begin{tikzpicture}[scale=1, x=0.5cm, y=0.5cm, every node/.style={transform shape}]
\tikzstyle{empty}=[draw, circle, minimum size=5pt, inner sep=0mm]
\tikzstyle{nero} =[draw, circle, minimum size=4pt, inner sep=0mm, fill]

\def\a{3}\def\b{7}\def\bp{11}\def\ap{14}
\def\t{0.3}
\def\yu{1.5}
\def\yb{-1.5}

\draw (-0.5,\yu)--(\ap+2.7,\yu);

\draw (0,\yu)--(0,\yb);
\draw (\a,\yu)--(\a,\yb);
\draw (\b,\yu)--(\b,\yb);
\draw (\bp,\yu)--(\bp,\yb);
\draw (\ap,\yu)--(\ap,\yb);

\node at (0,0) [nero] {};
\node at (0,-1) [nero] {};
\node at (\a,0) [nero] {};
\node at (\a,-1) [emptyn] {};
\node at (\b,0) [nero] {};
\node at (\b,-1) [emptyn] {};
\node at (\bp,0) [nero] {};
\node at (\bp,-1) [emptyn] {};
\node at (\ap,0) [nero] {};
\node at (\ap,-1) [nero] {};

\draw[arrows={triangle 45-triangle 45}] (\a+\t,0)--(\b-\t,0);
\draw[arrows={triangle 45-triangle 45}] (\b+\t,0)--(\bp-\t,0);
\draw[arrows={triangle 45-triangle 45}] (0+\t,-1)--(\b-\t,-1);
\draw[arrows={triangle 45-triangle 45}] (\bp+\t,-1)--(\ap-\t,-1);

\node at (\ap+1.95,0) {$\ldots$};
\node at (\ap+1.95,-1) {$\ldots$};


\node[left] at (0,-1) {$e$};
\node[left] at (\a,3pt) {$a$};
\node[above,right] at (\b,9pt) {$b$};
\node[right] at (\bp,3pt) {$b'$};
\node[right] at (\ap,-1) {\raisebox{3.5pt}{$a'$}};

\end{tikzpicture}
\end{center}
Observe
that between positions~$\beta$ and $\epsilon$ (inclusive) in $A$ there are $\ell+1$ beads.
Since $A$ represents an even partition, all gaps between these beads are even. Hence $\ell$ and
$\epsilon - \beta$ have the same parity and $\ell$ and $\epsilon - \beta + p$ have opposite parity.
Therefore the gap between beads $e$ and $b$ in $B$ is even if and only if 
$\ell$ is even, so if and only if $\nu$ is black.
If there is no bead after position $\beta+p$ in $B$ then it is now clear that $\nu$ is even
if and only if~$\nu$ is black. Suppose that bead $a'$
is the first bead after position $\beta + p$ in $B$. Let~$a'$ be in position~$\alpha'$.
Since $A$ has a gap between positions $\epsilon$ and
$\alpha'$, we see that
\[ \alpha' - \epsilon = (\alpha' - (\beta'+p)) + (\beta'-\beta) + \bigl( (\beta+p) - \epsilon) \bigr) \]
is odd. Hence $\alpha' - (\beta'+p)$ and $(\beta+p)-\epsilon$ have the same parity.
Thus the gaps between beads $b$ and $e$ and beads $b'$ and $a'$ in $B$
have the same parity. So again we have that $\nu$ is even if and only if $\nu$ is black.

\subsubsection*{Case 2: One bead} Suppose that bead $b$ is moved twice. If 
$\nu$ is even then, by Lemma~\ref{lemma:delta0}(i), bead $b$ is the final bead in $A$,
the leg-length of the rim $2p$-hook is $0$ and $\Delta(\nu) = 0$ and $\nu$ is black.

Conversely, suppose that $\nu$ is black. Let
$b$ be in position $\beta$ of $A$. Since $\delta(\nu) = 0$,
for each $j$ such that $1 \le j \le p-1$,
either there are beads in both positions $\beta+j$ and $\beta+j+p$,
or spaces in both these positions. Hence the leg-length of the rim $2e$-hook
corresponding to bead $b$ is even, and so is congruent to $0$ modulo $4$.
The number of beads in positions $\beta+1,\ldots, \beta+(p-1)$ of~$A$ is therefore even. 

Suppose, for a contradiction, that there is a bead after position $\beta$.
Let $e$ be the first bead after position $\beta$, and suppose that $e$ is in position $\epsilon$.
The positions $\beta, \ldots, \epsilon + p$ in $A$ begin and end with beads.
So we have evenly many beads in $A$ in 
\[ \bigl( (\epsilon+p)-\epsilon \bigr) + \bigl( \epsilon -\beta \bigr) + 1 \]
positions. Since $\epsilon - \beta$ is odd, this number is odd. Therefore
two of the beads in these positions
form an odd gap, a contradiction.  Hence $b$ is the final bead in $A$ and
$\nu$ is even.
\end{proof}

\section{Proof of Theorem~\ref{teo:vertices}}

In this section we completely characterize the vertices of all the indecomposable summands of $H^{(2^n)}$.
The following lemma is required.

\begin{lemma}\label{lemma:num}
Let $n \in \N$. There exist $k$, $\ell \in \N_0$ such that $\ell \le (k+1)(p-1)/2$ and
\[ \tag{$\star$} 2n = (k+1)\bigl(2+\frac{p-1}{2}k\bigr) + 2\ell. \]
\end{lemma}

\begin{proof}
For $k \in \N_0$ define 
$\theta_k = (k+1)\bigl( 2 + \frac{p-1}{2}k \bigr)$. Note that $\theta_k$
is even. Choose $k$ so that
$\theta_k \le 2n < \theta_{k+1}$ and then define $\ell$ so that ($\star$) holds.
Since  $\theta_{k+1} - \theta_{k} = 2 + (k+1)(p-1)$, we have $2\ell \le (k+1)(p-1)$, as required.
\end{proof}

\begin{proof}[Proof of Theorem~\ref{teo:vertices}]
By Proposition~\ref{prop:Green} it is sufficient to prove that every Foulkes module has a 
projective summand. In turn, by Proposition~\ref{prop:proj}(i), it is sufficient
to prove that for all $n$ there is an
even partition $2\lambda$ of $2n$ with $p$-core $\gamma$ such that 
$2n = |\gamma| + w(\gamma)p$, or, equivalently, such that
$2\lambda \in \mathcal{E}(\gamma)$.

We now construct such partitions. Let $k$ and $\ell$ be as in Lemma~\ref{lemma:num}.
Define
\[ \lambda_{k,0}=\bigl( 2+k(p-1),2+(k-1)(p-1), 
\ldots,2+(p-1),2\bigr). \]
Note that $\lambda_{k,0}$ is an even $p$-core partition of $(k+1)(2 + (p-1)k/2)$, represented
by an abacus having $k+1$ beads on runner $2$ and no beads on any other runner. Define
\[ \lambda_{k,\ell}=\lambda_{k,0}+\bigl( (2(s+1))^{r},(2s)^{k+1-r} \bigr) \]
where $s$ and $r$ are the unique natural numbers such that 
$0\leq r<k+1$, $0\leq s\leq \frac{p-1}{2}$ and
$\ell=(k+1)s+r$. Note 
that $\lambda_{k,\ell}$ is an even partition of $2n$. Let $\gamma$ be the $p$-core of $\lambda_{k,\ell}$.
We consider three cases.

\begin{itemize}
\item[(1)] If $r=0$ and $0 \le s < \frac{p-1}{2}$ then $\lambda_{k,\ell}$ is a $p$-core
partition represented by an abacus with $k+1$ beads on runner $2s+2$. Therefore $\lambda_{k,\ell}
= \gamma \in \mathcal{E}(\gamma)$.

\item[(2)] If $r=0$ and $s = \frac{p-1}{2}$ then $\lambda_{k,\ell}$ is represented by an abacus
with $k+1$ beads on runner $1$, in rows $1, \ldots, k+1$, and no beads on
any other runner. Now $\lambda_{k,\ell}$
is obtained from its $p$-core
$\gamma$ by the sequence of bead moves specified in Proposition~\ref{prop:core}(ii),
taking $t = k+1$, $t'=0$ and $u=1$. Therefore $\lambda_{k,\ell} \in \mathcal{E}(\gamma)$.

\item[(3)] If $r > 0$ then  $(k+1)s+r = \ell \le (k+1)(p-1)/2$ implies that $s < \frac{p-1}{2}$.
\begin{itemize}
\item[(a)] If $2s+2 < p-1$
then $\lambda_{k,\ell}$ is represented by an abacus with $k+1-r$ beads on runner $2s+2$
in rows $0$, \ldots, $k-r$ and $r$ beads on runner $2s+4$ in rows $k+1-r$, \ldots, $k$,
Now $\lambda_{k,\ell}$ is obtained from its $p$-core
$\gamma$ by the sequence of bead moves specified in
Proposition~\ref{prop:core}(i), taking $t= k+1-r$, $t' = r$, $u = 2s+2$ and $u' = 2s+4$.
Therefore $\lambda_{k,\ell} \in \mathcal{E}(\gamma)$. 
\item[(b)] If $2s+2 = p-1$
then $\lambda_{k,\ell}$ is represented by an abacus with $k+1-r$ beads on runner $p-1$
in rows $0$, \ldots, $k-r$ and $r$ beads on runner $1$ in rows $k+2-r$, \ldots, $k+1$.
Now $\lambda_{k,\ell}$ is obtained from its $p$-core
$\gamma$ by the sequence of bead moves specified in
Proposition~\ref{prop:core}(ii), taking $t= r$, $t' = k+1-r$, $u = p-1$ and $u' = 1$.
Therefore $\lambda_{k,\ell} \in \mathcal{E}(\gamma)$. 
\end{itemize}
\end{itemize}
This completes the proof of Theorem~\ref{teo:vertices}.
\end{proof}

\begin{remark}\label{remark:a>2}
Let $a$, $n \in \N$. Generalizing the Foulkes modules $H^{(2^n)}$ already defined,
let $H^{(a^n)}$ denote the $FS_{an}$-module induced
from the trivial representation of $S_a \wr S_n$. For $t \in \N_0$ let
$P_t$ be a Sylow $p$-subgroup of $S_a \wr S_{tp}$. 
In \cite{Giannelli} it is shown
that if $a < p$ and $U$ is an indecomposable summand
of $H^{(a^n)}$, then there exists $t \in \N_0$ such that $t \le n/p$ and
$P_t$ is a vertex of~$U$. It would be interesting to know if an analogue of Theorem \ref{teo:vertices} holds in this general setting. More precisely, is it true that for every $t\in \N_0$ such that $t \le n/p$, there is an 
indecomposable summand of $H^{(a^n)}$ with vertex $P_t$?
The main obstacle in proving this for arbitrary $a>2$ is the lack of knowledge of the 
ordinary character of the Foulkes module $H^{(a^n)}$.
\end{remark}

\section{Proof of Theorems~\ref{teo:lowweight} and~\ref{teo:maximalvertex}}

Throughout this section let $\gamma$ be a $p$-core. By reduction modulo $p$,
the composition factors of $H^{(2^n)}$ are precisely the composition factors of
the Specht modules $S^{2\lambda}$ for $\lambda$ a partition of $n$. A composition
factor of $S^{2\lambda}$ lies in the block $B(\gamma,w)$ if and only if $2\lambda \in B(\gamma,w)$.

\subsection*{Blocks of $p$-weight $0$}

The unique module in the block $B(\gamma,0)$ is the simple projective Specht module $S^\gamma$.
It is a composition factor of $H^{(2^n)}$ if and only if $\gamma$ is an even $p$-core. In this
case, since $S^\gamma$ is projective, it splits off as a direct summand. This proves
Theorem~\ref{teo:lowweight}(i).

\subsection*{Blocks of $p$-weight $1$}

Suppose that there is a summand in the block $B(\gamma,1)$. Then
an even partition can be obtained from $\gamma$ by adding a single $p$-hook.
Fix an abacus representing $\gamma$.
Suppose the highest odd gap is between
beads $b$ and $b'$ in positions $\beta$ and $\beta'$ where $\beta < \beta'$.
The only single-step moves that can lead to an even partition are moves
of $b$ and~$b'$. If $b'$ has a space below it then moving $b'$ gives
an even partition only if $b$ also has a space below it. A similar argument
applies if $b$ has a space below it. Therefore both $b$ and $b'$ have spaces below them
and a single-step move of either bead gives an even partition. The two possible
configurations are as shown in Figure 4 below.

\begin{figure}[h!]
\begin{center}
\makebox[\textwidth]{
\includegraphics{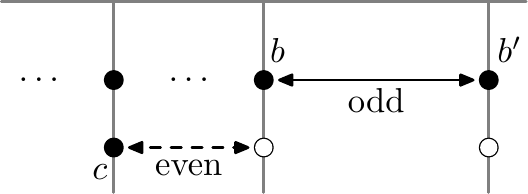}\hspace*{0.5in}
\includegraphics{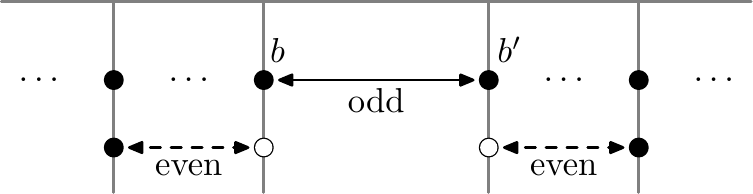}}
\end{center}
\caption{\small A $p$-core $\gamma$ such
that $w(\gamma) = 1$ has an abacus  of one of the forms shown above. 
In the left diagram, there is a single odd gap, bead $b'$ is on runner $p-1$,
 and the bead marked $c$ is the final
bead on the abacus. (It is possible that $c=b'$.) In the right diagram there are two odd gaps.
In each diagram a single-step move of either $b$ or $b'$ gives an even partition.}
\end{figure}

Let $2\lambda$ and $2\lambda'$ be the even partitions obtained by moving $b$
and $b'$, respectively.
Then $\mathcal{E}(\gamma) = \{2\lambda, 2\lambda'\}$.
It now follows from Propositions~\ref{prop:proj}(ii) and~\ref{prop:Cartan}
that $P^{2\lambda'}$ is the unique summand of $H^{(2^n)}$ in $B(\gamma, 1)$.
This completes the proof of Theorem~\ref{teo:lowweight}(iii).

\begin{remark}\label{remark:$p$-weight1}
By \cite[\S 6, Example 1]{Scopes}, $B(\gamma,1)$ is
Morita equivalent to the principal block of $FS_p$ by a Scopes functor.
Since there are no beads between $b$ and $b'$,
the partitions $2\lambda$ and $2\lambda'$ are neighbours in the dominance order
on partitions in $B(\gamma,1)$. Since Scopes functors preserve the dominance order on
partitions, it follows that the Specht factors of $P^{2\lambda'}$ 
are $S^{2\lambda'}$ (at the top) and $S^{2\lambda}$ (at the bottom). 
Thus $P^{2\lambda'}$ has the Loewy structure
\[ \begin{matrix} D^{2\lambda'} \\ D^\nu \oplus D^{2\lambda} \\ D^{2\lambda'}
\end{matrix} \]
where $\nu$ is the smallest partition greater than $2\lambda'$ in the dominance
order on partitions in $B(\gamma,1)$. (If $2\lambda'$ is greatest then omit $D^\nu$.)
\end{remark}

\subsection*{Blocks of $p$-weight $2$}

Suppose that there is a summand in the block $B(\gamma,2)$. Then either
$w(\gamma) = 2$, or $w(\gamma) = 0$ and $\gamma$ is an even $p$-core.

\subsubsection*{Case $w(\gamma) = 2$}
Let $A$ be an abacus for $\gamma$. Suppose that the adjacent beads~$b$ and $b'$ form
the highest odd gap in $A$. Since $\gamma$ is a $p$-core, we can assume (by adding initial beads to $A$) that $b$ and $b'$ lie in the same row $r$. Suppose, for a contradiction, that neither bead $b$
nor $b'$ is moved in a sequence of single-step bead moves leading to an abacus $B$
representing a partition in~$\mathcal{E}(\gamma)$. Then there exists a bead $c$, 
in an earlier position than bead $b$, that is moved into the gap between $b$ and $b'$.
Let $c'$ be the left adjacent bead to~$c$ and let $c''$ be the right adjacent bead to $c$.  
If $c$ is not in row $r-1$, then it is in row $r-2$ and is moved twice. This leaves 
an odd gap between beads $c'$ and~$c''$.
Hence~$c$ lies in row $r-1$. Let $A'$ be the abacus obtained from $A$ by a single-step move of bead $c$.
The odd gap between $c'$ and $c''$ in $A'$
cannot be removed by moving a bead $d$ from row $r-2$, since this creates a new odd
gap between the two beads adjacent to $d$ in $A'$. Hence either $c'$ or $c''$ is moved.
Therefore $B$ has two beads in the positions between $b$ and $b'$, and so there is an odd gap in $B$
involving $b$ or $b'$. 

We have shown that either bead $b$ or bead $b'$ is moved.
Suppose that there is a bead~$c$ immediately below bead $b'$. Then there must be
spaces in the positions immediately below bead $b$ and bead~$c$.
The partitions in $\mathcal{E}(\gamma)$ are obtained as follows:
\begin{itemize}
\item[(i)] two single-step moves
of bead $b$,
\item[(ii)] single-step moves of beads $b$ and $c$, 
\item[(iii)] a single-step move of bead $c$ followed by a single-step move of bead~$b'$. 
\end{itemize}
Hence $|\mathcal{E}(\gamma)| = 3$ and the unique maximal partition in $\mathcal{E}(\gamma)$
is the partition $2\lambda$ given by the moves in~(iii). By
Propositions~\ref{prop:proj}(ii) and~\ref{prop:Cartan}, $P^{2\lambda}$ is the
unique summand of $H^{(2^n)}$ in $B(\gamma, 2)$.
A similar result holds if there is a bead immediately below bead~$b$;
in this case the unique maximal partition is given by two single-step moves of bead $b'$.
(This case occurs in 
the block $B\bigl( (3,1), 2\bigr)$ when $p=3$ in the example in \S 8 below.)

Now suppose there are spaces below beads $b$ and $b'$. 
Let $\mu$ be the partition obtained by a single-step move of bead $b'$, represented
by the abacus $B$. Let $C$ be the abacus obtained by swapping the bead $b$
on $B$ with the space in $B$ in the position occupied by bead $b'$ in $A$.
Let $\gamma^\star$ be the $p$-core represented by $C$.
The abaci $B$ and $C$ have the same pattern of odd and even gaps. Moreover,
a single-step move of bead $b$ on $B$ does not give an even partition, since it 
restores the odd gap between beads $b$ and $b'$ present in $A$. Therefore 
$w(\gamma^\star) = 1$ and there is a bijection
between the sequences of single-step moves on $B$ and on $C$ that give even partitions.
From the $p$-weight one case we have $\mathcal{E}(\gamma^\star) = 2$. Hence exactly two 
even partitions can be obtained by starting with a single-step move of bead $b'$, leaving bead~$b$
fixed.
A similar argument deals with the case where bead $b$ is moved. Therefore
$|\mathcal{E}(\gamma)| = 4$, and by Propositions~\ref{prop:proj}(ii) and~\ref{prop:Cartan},
it follows that if $2\lambda$ is a maximal element of $\mathcal{E}(\gamma)$ then
$P^{2\lambda}$ is the unique summand of $H^{(2^n)}$ in $B(\gamma, 2)$.

\begin{example}
We pause to give an example of the case where $w(\gamma) = 2$
and there are spaces below beads $b$ and $b'$. 
Let $p=5$ and let $\gamma = (5,1,1,1)$.
An abacus $A$ representing $\gamma$, and the abacus $C$ defined in the proof, are shown
below left and below right.

\begin{center}
\begin{tikzpicture}[scale=0.75, x=1cm, y=0.5cm, every node/.style={transform shape}]
\tikzstyle{empty}=[draw, circle, minimum size=5pt, inner sep=0mm]
\tikzstyle{nero} =[draw, circle, minimum size=4pt, inner sep=0mm, fill]

\node at (-0.8,-0.5) {$A$};
\draw (-0.5,0.5)--(4.5,0.5);
\node at (0,0) [nero] {};  \node[yshift=15pt] at (0,0) {$\scriptstyle 0$};
\node at (1,0) [empty] {}; \node[yshift=15pt] at (1,0) {$\scriptstyle 1$};
\node at (2,0) [nero] {}; \node[yshift=15pt] at (2,0) {$\scriptstyle 2$};
\node at (3,0) [nero] {}; \node[yshift=15pt] at (3,0) {$\scriptstyle 3$};
\node at (4,0) [nero] {}; \node[yshift=15pt] at (4,0) {$\scriptstyle 4$};
\node at (0,-1) [empty] {};
\node at (1,-1) [empty] {};
\node at (2,-1) [empty] {};
\node at (3,-1) [empty] {};
\node at (4,-1) [nero] {};
\node at (0,-2) [empty] {};
\node at (1,-2) [empty] {};
\node at (2,-2) [empty] {};
\node at (3,-2) [empty] {};
\node at (4,-2) [empty] {};
\node[left] at (0,0) {$b$};
\node[right] at (2,0) {$b'$};

\end{tikzpicture}
\hspace*{0.5in}
\begin{tikzpicture}[scale=0.75, x=1cm, y=0.5cm, every node/.style={transform shape}]
\tikzstyle{empty}=[draw, circle, minimum size=5pt, inner sep=0mm]
\tikzstyle{nero} =[draw, circle, minimum size=4pt, inner sep=0mm, fill]

\node at (-0.8,-0.5) {$C$};
\draw (-0.5,0.5)--(4.5,0.5);
\node at (0,0) [empty] {};  \node[yshift=15pt] at (0,0) {$\scriptstyle 0$};
\node at (1,0) [empty] {}; \node[yshift=15pt] at (1,0) {$\scriptstyle 1$};
\node at (2,0) [nero] {}; \node[yshift=15pt] at (2,0) {$\scriptstyle 2$};
\node at (3,0) [nero] {}; \node[yshift=15pt] at (3,0) {$\scriptstyle 3$};
\node at (4,0) [nero] {}; \node[yshift=15pt] at (4,0) {$\scriptstyle 4$};
\node at (0,-1) [empty] {};
\node at (1,-1) [empty] {};
\node at (2,-1) [nero] {};
\node at (3,-1) [empty] {};
\node at (4,-1) [nero] {};
\node at (0,-2) [empty] {};
\node at (1,-2) [empty] {};
\node at (2,-2) [empty] {};
\node at (3,-2) [empty] {};
\node at (4,-2) [empty] {};

\node[right] at (2,0) {$b$};
\node[right] at (2,-1) {$b'$};
\node[right] at (4,-1) {$d$};

\end{tikzpicture}

\end{center}
Moving bead $b'$ or bead $d$ in $C$ gives an even partition. 
The corresponding elements of $\mathcal{E}(\gamma)$, obtained by 
moving bead $b$ in $C$ back to position $0$ 
are
$(8,6,2,2)$ and $(10,4,2,2)$, respectively. The other partitions
in $\mathcal{E}(\gamma)$ are found by moving $b$ first; they
are $(6,6,2,2,2)$ and $(10,2,2,2,2)$. Thus the unique summand
of $H^{(2^{9})}$ in $B(\gamma, 2)$ is $P^{(10,4,2,2)}$.
\end{example}

\subsubsection*{Case $w(\gamma) = 0$}
Since $H^{(2^{n-p})}$ has $S^\gamma$ as its unique summand in the block $B(\gamma,0)$,
it follows from Proposition~\ref{prop:Green} that there is a unique summand of $H^{(2^n)}$ 
in $B(\gamma,2)$
with vertex $Q_1$.
 Since any summand in $B(\gamma,2)$ has a
vertex contained in the $p$-weight $2$ defect group $\langle (1,2,\ldots, p), (p+1,\ldots, 2p) \rangle$
(see \cite[Theorem 6.2.45]{JK}), any other summand in this block must be projective.

Suppose that $P^\nu$ is such a projective summand. By Proposition~\ref{prop:proj}(iii)
$\nu$ is an even partition.  By Lemma~\ref{lemma:delta0}(ii),
$\delta(\nu) = 0$. By Theorem~\ref{teo:Richards} and Proposition~\ref{prop:Cartan},
 the column of the decomposition
matrix labelled by $\nu$ has a non-zero entry in a row labelled by partition 
$\mu$ with $\delta(\mu) = 1$. By another application of Lemma~\ref{lemma:delta0}(ii),
this partition $\mu$ is not even, contradicting Proposition~\ref{prop:proj}(iii).
This completes the proof.
$\qed$
\vspace{.2cm}

As a corollary, we are now ready to deduce Theorem \ref{teo:maximalvertex}.

\begin{proof}[Proof of Theorem~\ref{teo:maximalvertex}]

Let $t = \lfloor n /p \rfloor$ and let $r = n-tp$. Since $Q_t$ permutes $2tp$
points, if a block $B(\gamma,w)$ of $S_{2n}$ contains a summand of $H^{(2^n)}$
with vertex $Q_t$ then $wp \ge 2tp$, and so $w = 2t$.
Let $\gamma$ be a $p$-core such that $|\gamma| + 2tp = 2n$.
By Proposition~\ref{prop:Green}, the number of indecomposable summands of $H^{(2^n)}$
with vertex $Q_t$ in $B(\gamma,2t)$ is equal to the
number of indecomposable projective summands of 
$H^{(2^r)}$ in the block $B(\gamma,w-2t)$ of $S_{2r}$. Since $r < p$,
it follows from Theorem~\ref{teo:lowweight} that $H^{(2^r)}$ has at most
one summand in $B(\gamma,w-2t)$, and that every such summand is projective.
Hence the number of indecomposable summands of $H^{(2^n)}$ with 
vertex $Q_t$ is equal to the number of blocks of $H^{(2^r)}$ containing an even 
partition. This equals the number of $p$-core partitions that
can be obtained by removing $p$-hooks from even partitions of $2r$,
as required.
\end{proof}

\section{Proof of Theorem~\ref{teo:Scott}}

Let $2n = 2k+2p$. By hypothesis $0 \le 2k < p$. The $p$-core $(2k)$ 
is represented by the abacus shown below. (If $k=0$ then the bead in the
second row should be deleted.)

\smallskip
\begin{center}
\begin{tikzpicture}[scale=1, x=0.5cm, y=0.5cm, every node/.style={transform shape}]
\tikzstyle{empty}=[draw, circle, minimum size=5pt, inner sep=0mm]
\tikzstyle{nero} =[draw, circle, minimum size=4pt, inner sep=0mm, fill]
\def\x{4.5}\def\xx{2.5}
\def\y{10.5}\def\yy{8.7}
\draw (-0.5,0.5)--(12,0.5);
\node at (0,0) [nero] {};
\node at (1,0) [nero] {};
\node at (\x,0) [nero] {};
\node at (\x+1,0) [nero] {};
\node at (\x+2,0) [nero] {};
\node at (\y,0) [nero] {};
\node at (\y+1,0) [nero] {};

\node at (0,-1) [empty] {};
\node at (1,-1) [empty] {};
\node at (\x,-1) [empty] {};
\node at (\x+1,-1) [nero] {};
\node at (\x+2,-1) [empty] {};
\node at (\y,-1) [empty] {};
\node at (\y+1,-1) [empty] {};

 \node[yshift=15pt] at (0,0) {$\scriptstyle 0$};
\node[yshift=15pt] at (1,0) {$\scriptstyle 1$};
\node[yshift=15pt] at (\x-0.25,0) {$\scriptstyle 2k-1$};
\node[yshift=15.5pt] at (\x+1,0) {$\scriptstyle 2k$};
\node[yshift=15pt] at (\x+2.25,0) {$\scriptstyle 2k+1$};
\node[yshift=15pt] at (\y,0) {$\scriptstyle p-2$};
\node[yshift=15pt] at (\y+1.5,0) {$\scriptstyle p-1$};

\node at (\xx,-0.5) {$\cdots$};
\node at (\yy,-0.5) {$\cdots$};
\node[yshift=15pt] at (\xx,0) {$\cdots$};
\node[yshift=15pt] at (\yy,0) {$\cdots$};

\end{tikzpicture}
\end{center}
Let $\ob{u}$ denote the partition obtained by
two single-step moves of the lowest bead on runner $u$. Let $\tb{u}{u}$
denote the partition obtained by
a single-step move of the lowest bead on runner $u$ followed by a single-step move
of the bead immediately above it. Finally let $\tb{u}{v}$ denote the partition obtained by
single-step moves of the
lowest beads on runners $u$ and $v$. 

It follows from Proposition~\ref{prop:Delta0},
but can also easily be seen directly, that if $k \not= 0$ then
the even partitions in $B\bigl( (2k),2 \bigr)$ are $\ob{2k}$, $\ob{p-1}$ and
$\tb{j}{j+1}$ where  \emph{either} $j < 2k$ and $j$ is even, 
\emph{or} $j > 2k$ and $j$ is odd. 
If $k=0$ then they are $\ob{p-1}$ and $\tb{j}{j+1}$ for $j$ even.
A convenient way to find the composition factors of the corresponding
Specht modules uses the chains in the following lemma.

\begin{lemma}\label{lemma:chains}
Let $\mathcal{P}$ be the set consisting of all even partitions
in $B\bigl( (2k),2\bigr)$ 
together with all partitions $\nu \in B\bigl( (2k), 2\bigr)$ such that $\delta(\nu) = 1$.
Then $\mathcal{P}$ is totally ordered by the dominance order.
The elements of $\mathcal{P}$ are as follows. If $2 < 2k < p-1$ then
\begin{align*}
{}&{} \obb{2k} \rhd \ob{p-1} \rhd \tb{2k}{p\m2} \rhd \tbb{p-2}{p-1} \rhd \cdots \\
& \cdots \rhd \tbb{2t\p1}{2t\p2} \rhd \tb{2t}{2t\p2} \rhd \tb{2t\m1}{2t\p1} \rhd \tbb{2t\m1}{2t} \rhd \cdots
\\
& \cdots \rhd \tbb{2k\p1}{2k\p2} \rhd \tb{2k\m1}{2k\p2} \rhd \tb{2k\m2}{2k\p1} \rhd
\tbb{2k\m2}{2k\m1} \rhd \cdots \\ 
& \cdots \rhd \tbb{2s}{2s\p1} \rhd \tb{2s\m1}{2s\p1} \rhd \tb{2s\m2}{2s} \rhd \tbb{2s\m2}{2s\m1} \rhd \cdots
\\
& \hspace*{2.1in} 
\cdots \rhd \tbb{2}{3} \rhd \tb{1}{3} \rhd \tb{0}{2} \rhd \tbb{0}{1} \rhd \tb{1}{1}; 
\intertext{if $2k = 0$ then}
&\obb{p-1} \rhd \ob{p-2} \rhd \tb{p-3}{p-1} \rhd \tbb{p-3}{p-2} \rhd \tb{p-4}{p-2} \rhd \cdots \\
& \hspace*{1.575in} \cdots \rhd 
\tb{2}{4} \rhd \tbb{2}{3} \rhd \tb{1}{3} \rhd \tb{0}{2} \rhd \tbb{0}{1} \rhd \tb{1}{1}
\intertext{with the general case as above; if $2k=2$ then}
&\obb{2} \rhd \ob{p-1} \rhd \tb{2}{p-2} \rhd \tbb{p-2}{p-1} \rhd \tb{p-3}{p-1} \rhd \cdots \\
& \hspace*{1.575in} \cdots \rhd 
\tb{3}{5} \rhd \tbb{3}{4} \rhd \tb{1}{4} \rhd \tb{0}{3} \rhd \tbb{0}{1} \rhd \tb{1}{1} 
\intertext{with the general case as above; and if $2k=p-1$ then}
&\obb{p-1} \rhd \tb{p-2}{p-1} \rhd \ob{p-3} \rhd \tbb{p-3}{p-2} \rhd \tb{p-4}{p-2} \rhd \cdots \\
& \hspace*{1.575in} \cdots \rhd \tb{2}{4} 
\rhd \tbb{2}{3} \rhd \tb{1}{3} \rhd \tb{0}{2} \rhd \tbb{0}{1} \rhd \tb{1}{1} \end{align*}
with the general case as above. The even partitions are shown in bold type.
\end{lemma}

\begin{proof}
It is routine to check that the partitions $\nu$ such that $\delta(\nu) = 1$ are as claimed.
Lemma 4.4 in \cite{Richards}, which states that $\tb{u}{v} \unlhd \tb{u'}{v'}$ if and only if
$u \le u'$ and $v \le v'$, then gives the total order of the chains, except
for the cases involving partitions of the form $\ob{u}$, which have to be checked separately.
\end{proof}

\begin{proposition}\label{prop:Scott}
Let $\mu \in B\bigl( (2k), 2 \bigr)$ be even. 
Suppose that, in the relevant chain of partitions in Lemma~\ref{lemma:chains},
we have adjacent partitions $\mu' \rhd \nu' \rhd \nu \rhd \mu \rhd \nu''$ where $\mu'$ is even. Then
\begin{thmlist}
\item $\mu$ is even and ${\mu'}^\circ = \mu$ and ${\nu}^\circ = \nu''$;
\item if $\mu \not= \langle 0,1 \rangle$ then
the composition factors of $S^\mu$ are $D^{\mu'}$, $D^\nu$ and $D^{\mu}$;
\item the composition factors of $S^{\tb{0}{1}}$ are 
$D^{\tb{2}{3}}$ and $D^{\tb{0}{2}}$;
\item $\Ext^1(D^{\mu'}, D^\nu) = \Ext^1(D^\nu,D^{\mu}) = F$;
\item if $\lambda \in B\bigl( (2k), 2\bigr)$ and $\Ext^1(D^\mu, D^\lambda) \not=0$
then $\delta(\lambda) = 1$.
\end{thmlist}
\end{proposition}

\begin{proof}
Part (i) follows from the definition of the map $\circ$ and inspection of the chains in
Lemma~\ref{lemma:chains}. Then
(ii) and (iii) are easy deductions from Lemma~\ref{lemma:chains} and Theorem~\ref{teo:Richards}.
The Ext quivers of $p$-weight two principal blocks
of symmetric groups were found by Martin in~\cite{MartinPhD}: parts (iv) and (v) 
can be read off from Figures 9 and 10 in the appendix. 
\end{proof}

We are now ready to prove Theorem~\ref{teo:Scott}.

\begin{proof}[Proof of Theorem~\ref{teo:Scott}]
By Theorem~\ref{teo:lowweight}(iii), there is a unique indecomposable summand of $H^{(2^{p+k})}$
in the block $B\bigl( (2k), 2\bigr)$ of $S_{2(p+k)}$. This summand,~$U$ say,
is the Scott module of vertex $Q_1$. By \cite[(2.1)]{BroueBrauer}, $U$ is self-dual
and has the trivial module in its socle. By Proposition~\ref{prop:Scott}(ii) and (iii),
each $D^\nu$ labelled by a partition $\nu$ such that $\delta(\nu) = 1$ appears exactly
once in $U$. Hence if $2 < 2k < p-1$ then the heart of $U$ is
\[\begin{split} D^{\tb{2k}{p-2}} {}&{}\oplus D^{\tb{p-4}{p-2}} \oplus \cdots \\
&\cdots \oplus D^{\tb{2k+1}{2k+3}} 
\oplus D^{\tb{2k-2}{2k+1}}
\oplus D^{\tb{2k-4}{2k-2}} \oplus \cdots \oplus D^{\tb{0}{2}}.\end{split} \]
If $2k \in \{0,2,p-1\}$ then an analogous result holds with minor changes to the
labels, as indicated in Figure~5 below.
Similarly, each simple module $D^{2\lambda}$ labelled by a $p$-regular even partition 
$2\lambda$ appears exactly
twice in~$U$. 
By Proposition~\ref{prop:Scott}(iv), the only extensions in $U$
are between modules in the heart of $U$ and these $D^{2\lambda}$. 
It is easily seen, either from the chains in Lemma~\ref{lemma:chains},
or from Lemma~4.3 in \cite{Richards} and our Lemma~\ref{lemma:delta0}, 
that the unique even partition in $B\bigl( (2k), 2\bigr)$
that is not $p$-regular is $(2k,2^p)$ (interpreted as $(2^p)$ if $k=0$).
Hence $U$ has 
three Loewy layers and
\[ \soc U \cong \top U \cong \bigoplus_{2 \lambda \in B( (2k), 2) \atop 2\lambda\not= (2k, 2^p)} D^{2\nu}. \]
If $2 < 2k < p-1$ then the structure of $U$ is as shown in Figure 1.
The exceptional cases are shown in Figure~5 below.
\end{proof}

\begin{figure}[h]
\begin{center}
\makebox[\textwidth]{
\begin{tikzpicture}[scale=1, x=1cm, y=1cm]
\tikzstyle{nero} = [draw, circle, minimum size=4pt, inner sep=0mm, fill]

\renewcommand{\scriptstyle}[1]{\scalebox{0.9}{$#1$}}
\def\x{6.3}\def\v{0}\def\ww{1.5}
\newcommand{\s}[1]{\scalebox{0.9}{#1}}

\foreach \w in {0,\x} {
	\foreach \z in {0,-\ww-\ww} {
		\draw (0+\w,\z)--(2+\w,\z-2);
		\draw (0+\w,\z-2)--(2+\w,\z);
	}
}

\foreach \z in {0,0-\ww-\ww} {
	\draw (2,\z)--(3.5,\z-1.5);
	\draw (2,\z-2)--(3.5,\z-0.5);
	\draw (\x,\z)--(\x-1.5,\z-1.5);
	\draw (\x,\z-2)--(\x-1.5,\z-0.5);
	\draw (\x+2,\z)--(\x+3,\z-1);
	\draw (\x+2,\z-2)--(\x+3,\z-1);
}

\foreach \z in {0,-2} {
	\node at (0,\z) [fill=white] {\s{$D^{\ob{p-1}}$}};
	\node at (2,\z) [fill=white] {\s{$D^{\tb{p-3}{p-2}}$}};
	\node at (\x,\z) [fill=white] {\s{$D^{\tb{4}{5}}$}};
	\node at (\x+2,\z) [fill=white] {\s{$D^{\tb{2}{3}}$}};
	\node at (4.25,\z) {$\cdots$};
}


\foreach \z in {0-\ww-\ww,-2-\ww-\ww} {
	\node at (0,\z) [fill=white] {\s{$D^{\ob{2}}$}};
	\node at (2,\z) [fill=white] {\s{$D^{\tb{p-2}{p-1}}$}};
	\node at (\x,\z) [fill=white] {\s{$D^{\tb{5}{6}}$}};
	\node at (\x+2,\z) [fill=white] {\s{$D^{\tb{3}{4}}$}};
	\node at (4.25,\z) {$\cdots$};
}

\node at (1,-1) [fill=white] {\s{$D^{\tb{p-3}{p-1}}$}};
\node at (3,-1) [fill=white] {\s{$D^{\tb{p-5}{p-3}}$}};
\node at (\x-1,-1) [fill=white] {\s{$D^{\tb{4}{6}}$}};
\node at (\x+1,-1) [fill=white] {\s{$D^{\tb{2}{4}}$}};
\node at (\x+3,-1) [fill=white] {\s{$D^{\tb{0}{2}}$}};


\node at (1,-1-\ww-\ww) [fill=white] {\s{$D^{\tb{2}{p-2}}$}};
\node at (3,-1-\ww-\ww) [fill=white] {\s{$D^{\tb{p-4}{p-2}}$}};
\node at (\x-1,-1-\ww-\ww) [fill=white] {\s{$D^{\tb{5}{7}}$}};
\node at (\x+1,-1-\ww-\ww) [fill=white] {\s{$D^{\tb{3}{5}}$}};
\node at (\x+3,-1-\ww-\ww) [fill=white] {\s{$D^{\tb{0}{3}}$}};

\end{tikzpicture}}
\end{center}

\caption{\small Loewy layers of the Scott module summand $U$ of $H^{(2n)}$ 
in the block $B\bigl( (2k), 2\bigr)$ in the cases $2k=0$ (top) and $2k=2$ (bottom).
If $2k = p-1$ then the diagram for $2k=0$ applies, replacing $\tb{p-3}{p-1}$
with~$\ob{p-3}$.}
\end{figure}
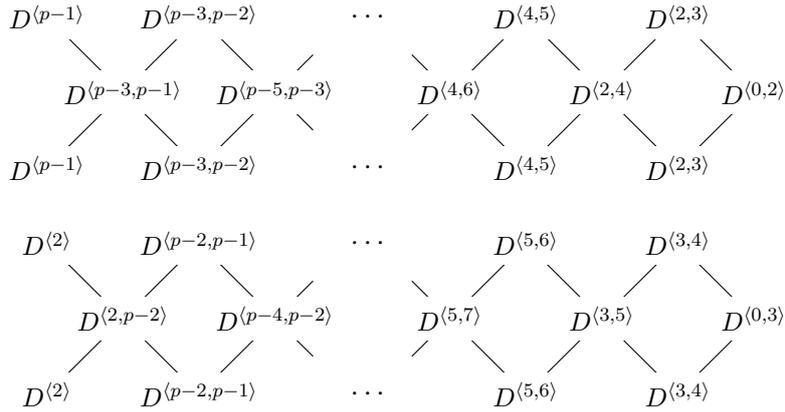

\begin{remark}\label{remark:Paget}
In \cite{PagetFiltration}, Paget proved that $H^{(2^n)}$ has a  filtration
$0 = V_0 \subset V_1 \subset \ldots \subset V_d = H^{(2^n)}$ such that
$V_i/V_{i-1} \cong S^{2\lambda(i)}$, where
$\lambda(1) < \ldots < \lambda(d)$ are the partitions of $n$. By Proposition 13 in \cite{WildonMFs}, which
was proved using the Hemmer--Nakano homological characterization of modules with a Specht filtration
(see \cite{HemmerNakano}), it follows that the summand $U$ has a Specht filtration, provided $p \ge 5$.
This filtration can be seen in Figures~1 and~5: for $2k > 2$, the
bottom Specht factor is $S^{\tb{0}{1}}$  (with composition factors $D^{\tb{0}{2}}$ 
and $D^{\tb{2}{3}}$) the next is $S^{\tb{2}{3}}$ (with composition factors $D^{\tb{2}{3}}$,
$D^{\tb{2}{4}}$, $D^{\tb{4}{5}}$), and on on, ending with the trivial module $D^{\ob{2k}}$
at the top.
\end{remark}

It would be interesting to know other cases where a single indecomposable module exhibits a significant
proportion of the extensions between two classes of simple modules for the symmetric group.

\section{Example: $H^{(2^7)}$ in characteristic $5$}

Let $p=5$. The even partitions of $14$ lie in the $5$-blocks $B\bigl( (4), 2 \bigr)$,
$B\bigl( (2,2), 2 \bigr)$, $B\bigl( (3,1), 2 \bigr)$, $B\bigl( (1^4), 2 \bigr)$,
$B\bigl( (5,2,2), 1 \bigr)$ and $B\bigl( (4,4,2,2,2), 0 \bigr)$ of $S_{14}$.
The principal block summand, $U$ say, is dealt with by Theorem~\ref{teo:Scott}.
Since $w\bigl( (3,1)\bigr) = w\bigl( (1^4) \bigr) = 2$, $w\bigl( (5,2,2) \bigr) = 1$ and 
$w\bigl( (4,4,2,2,2) \bigr) = 0$, it follows
from Theorem~\ref{teo:lowweight} that the summands
in these blocks are the projective modules $P^{(10,4)}$, $P^{(8,2,2,2)}$, $P^{(10,2,2)}$
and $S^{(4,4,2,2,2)}$,
respectively.
By Theorem~\ref{teo:lowweight}(iii), there is a unique summand, $V$ say, 
in $B\bigl( (2,2), 2\bigr)$.
This summand has vertex~$Q_1$. By Remark~\ref{remark:Paget}, $V$
has a Specht filtration by the Specht factors in this block, namely $S^{(12,2)}$, $S^{(6,4,4)}$
and $S^{(2^7)}$. 
Using Theorem~\ref{teo:Richards} to get the required decomposition numbers,
one finds that
\newcommand{\Sb}[1]{\scalebox{1}{$\displaystyle #1$}}
\newcommand{\Sbb}[1]{\scalebox{0.9}{$\displaystyle #1$}}
\begin{align*}\allowdisplaybreaks[4]
 U{}&{}  \oplus V \oplus P^{(10,4)} \oplus P^{(8,2,2,2)} 
\\
&=
\Sb{\left( \begin{matrix} S^{(14)} \\ S^{(4,4,4,2)} \\ S^{(4,2^5)} \end{matrix} \right)}
\bigoplus
\Sb{\left( \begin{matrix} S^{(12,2)} \\ S^{(6,4,4)} \\ S^{(2^7)} \end{matrix} \right)}
\bigoplus
\Sb{\left( \begin{matrix} S^{(10,4)} \\ S^{(8,6)} \\ S^{(8,4,2)} \end{matrix} \right)} 
\bigoplus
\Sb{\left( \begin{matrix} S^{(8,2,2,2)} \\ S^{(6,4,2,2)} \\ S^{(6,2^4)} \end{matrix} \right)} \\
&= 
\Sbb{\left( \begin{matrix} D^{(14)}  \oplus D^{(4,4,4,2)} \\ D^{(7,5,1,1)} \oplus D^{(4,3,2,2,2,1)}
\\ D^{(14)} \oplus D^{(4,4,4,2)} \end{matrix} \right)}
\bigoplus
\Sbb{\left( \begin{matrix} D^{(12,2)}  \oplus D^{(6,4,4)} \\ D^{(7,4,3)} \oplus D^{(3,3,3,2,1,1,1)}
\\ D^{(12,2)} \oplus D^{(6,4,4)} \end{matrix} \right)} 
\\&\hspace*{0.55in}
\bigoplus
\Sbb{\left( \begin{matrix} D^{(10,4)} \\ D^{(13,1)} \oplus D^{(8,6)} \\ D^{(10,4)} \oplus D^{(8,4,2)} \\
D^{(13,1)} \oplus D^{(8,6)} \\ D^{(10,4)} \end{matrix} \right)}
\bigoplus
\Sbb{\left( \begin{matrix} D^{(8,2,2,2)} \\ D^{(9,2,2,1)} \oplus D^{(6,4,2,2)} \\ 
D^{(6,5,2,1)} \oplus
D^{(8,2,2,2)} \oplus D^{(6,2,2,2,2)} \\ D^{(9,2,2,1)} \oplus  D^{(6,4,2,2)} \\ D^{(8,2,2,2)} \end{matrix}
\right) }.
\end{align*}
The Loewy layers of $P^{(10,2,2)}$ are given by Remark 6.1, while the other summand, $S^{(4,4,2,2,2)}$
is simple. By Proposition~6.5 in \cite{GiannelliWildonDec}, all these summands
have abelian endomorphism rings. We note that $P^{(10,4)}$ is the projective
summand used in the proof of Theorem~\ref{teo:vertices}.

%
\def\cprime{$'$} \def\Dbar{\leavevmode\lower.6ex\hbox to 0pt{\hskip-.23ex
  \accent"16\hss}D} \def\cprime{$'$}
\providecommand{\bysame}{\leavevmode\hbox to3em{\hrulefill}\thinspace}

\providecommand{\href}[2]{#2}

\end{document}